\documentclass[11pt]{article}
\usepackage{amsmath,amsthm,amsfonts,amssymb,mathrsfs,bm}
\usepackage{amsmath,amsthm,amsfonts,amssymb,bm,wasysym}
\usepackage{subfigure}
\usepackage{graphicx}
\usepackage[usenames]{color}
\usepackage{verbatim}

\parskip \medskipamount
\newtheorem{theorem}{Theorem}[section]
\newtheorem{definition}[theorem]{Definition}

\numberwithin{equation}{section}
\newtheorem{lemma}[theorem]{Lemma}

\newtheorem{corollary}[theorem]{Corollary}

\newtheorem{remark}[theorem]{Remark}

\newtheorem{claim}[theorem]{Claim}

\numberwithin{equation}{section}

\def\N{\mathbb{N}}

\def\Q{\mathcal{Q}}
\def\R{\mathbb{R}}

\def\EE{\mathcal{E}}
\def\B{\mathcal{B}}

\def\D{\mathcal{D}}

\def\H{\mathcal{H}}
\renewcommand{\phi}{\varphi}
\renewcommand{\epsilon}{\varepsilon}

\allowdisplaybreaks

\newcommand{\1}{{\text{\Large $\mathfrak 1$}}}

\newcommand{\til}{\widetilde}

\newcommand{\pr}[1]{\mathbb{P}\!\left(#1\right)}
\newcommand{\E}[1]{\mathbb{E}\!\left[#1\right]}

\newcommand{\norm}[1]{\left\| #1 \right\|}

\newcommand{\tn}{|\kern-.1em|\kern-0.1em|}

\newcommand{\gra}[2]{\mathrm{Gr}_{#2}(#1)}
\newcommand{\ima}[2]{\mathcal{R}_{#2}(#1)}

\newcommand{\gr}[1]{\mathrm{Gr}(#1)}
\newcommand{\im}[1]{\mathcal{R}(#1)}

\newcommand{\ig}{I_{\gamma,H}}
\newcommand{\cp}{\mathrm{Cap}}
\newcommand{\dpsi}[1]{\dim_{\Psi,H}(#1)}
\newcommand{\dmin}[1]{\dim_{\mathcal{M}}(#1)}
\newcommand{\cpc}[2]{\mathrm{Cap}_{#1}(#2)}

\newcommand\be{\begin{equation}}
\newcommand\ee{\end{equation}}

\begin{document}
\title{\bf Dimension of Fractional Brownian motion \\with variable drift}

\author{
Yuval Peres\thanks{Microsoft Research, Redmond, Washington, USA; peres@microsoft.com} \and Perla Sousi\thanks{University of Cambridge, Cambridge, UK;   p.sousi@statslab.cam.ac.uk}
}
\maketitle
\begin{abstract}
Let $X$ be a fractional Brownian motion in $\R^d$. For any Borel function $f:[0,1] \to \R^d$, we express the Hausdorff dimension of the image and the graph of~$X+f$ in terms of~$f$. This   is new even for the case of Brownian motion and continuous $f$, where it was known that this dimension is almost surely constant. The expression involves an adaptation of the parabolic dimension previously used by Taylor and Watson to characterize polarity for the heat equation. In the case when the graph of~$f$ is a self-affine McMullen-Bedford carpet, we obtain an explicit formula  for the dimension of the graph of~$X+f$ in terms of the generating pattern. In particular, we show that it can be strictly bigger than the maximum of the Hausdorff dimension of the graph of~$f$ and that of~$X$. Despite the random perturbation, the Minkowski and Hausdorff dimension of the graph of~$X+f$ can disagree.
\newline
MSC 2010 \emph{subject classifications.}
Primary   60J65.
\end{abstract}

\section{Introduction}
Let $B$ denote standard $d$-dimensional Brownian motion and suppose that  $f:[0,1] \to \R^d$ is continuous. Our main goal in this paper is to answer the following three questions.
\begin{enumerate}
\item In \cite{PS10} the authors showed that the Hausdorff dimension $\dim \gr{B+f}$  of the Graph of $B+f$ is almost surely constant. How can this constant be determined explicitly from $f \,  $?
\item Let $d=1$. Is there a continuous function $f$ such that the inequality established in \cite{PS10},
$\dim (\gr{B+f}) \ge \max\{\dim (\gr{B}), \dim (\gr{f}) \}$  a.s., is strict? \newline
For Minkowski ($=$ Box) dimension $\dim_M$ in place of Hausdorff dimension, the corresponding inequality is an equality for all continuous $f$, see \cite{CPS12}.
\item Falconer~\cite{falconer88} and Solomyak~\cite{solomyak98} showed that for almost all parameters in the construction of a self-affine set $K$, the Hausdorff dimension $\dim K$ and the Minkowski dimension $\dim_M K$ coincide. Earlier, McMullen~\cite{McMullen} and Bedford~\cite{Bedford_thesis} exhibited a special class of self-affine sets $K$ with  $\dim (K)<\dim_M (K)$. Is this strict inequality robust under some class of perturbations, at least when $K$ is the graph of a function?
\end{enumerate}
We will study fractal properties of graphs and images in a more general setting.
Let $X$ be a fractional Brownian motion in~$\R^d$ and $f$ a Borel measurable function.  We will express the dimension of both the image and the graph of~$X+f$ in terms of the so-called parabolic Hausdorff dimension of the graph of~$f$, which was first introduced by Taylor and Watson in~\cite{TaylorWatson} in order to determine polar sets for the heat equation. We start by introducing some notation and then give the definition of the parabolic Hausdorff dimension.

For a function~$h:[0,1]\to \R^d$ we denote by~$\gra{h}{A}=\{(t,h(t): t\in A\})$ the graph of~$h$ over the set~$A$ and by~$\ima{h}{A} = \{h(t): t\in [0,1]\}$ the image of~$A$ under~$h$. We write simply $\gr{h}=\gra{h}{[0,1]}$.

\begin{definition}\rm{
Let $A\subseteq \R_+  \times \R^d$ and $H\in [0,1]$. For all $\beta>0$ the $H$-parabolic $\beta$-dimensional Hausdorff content is defined by
\[
\Psi^\beta_H(A) = \inf\left\{ \sum_{j}\delta_j^\beta: \ A \subseteq \cup_j [a_j,a_j+\delta_j] \times [b_{j,1}, b_{j,1} + \delta_j^H] \times \ldots \times [b_{j,d},b_{j,d} + \delta_j^H] \right\},
\]
where the infimum is taken over all covers of $A$ by rectangles  of the form given above.
The $H$-\textit{parabolic Hausdorff dimension} is then defined to be
\[
\dpsi{A} = \inf\{ \beta: \Psi^\beta_H(A) =0\}.
\]
}
\end{definition}

This was introduced for $H=1/2$ by Taylor and Watson~\cite{TaylorWatson} in their study of polar sets for the heat equation.

We are now ready to state our main result which gives the dimension of the graph and the image of~$X+f$ in terms of~$\dpsi{\gr{f}}$. 
We write $\dim(A)$ for the Hausdorff dimension of the set~$A$.

\begin{theorem}\label{thm:dim-graph-image}
Let $X$ be a fractional Brownian motion in $\R^d$ of Hurst index~$H$, let~$f:[0,1]\to \R^d$ be a Borel measurable function and~$A$ a Borel subset of~$[0,1]$. If~$\alpha = \dpsi{\gra{f}{A}}$, then almost surely
\begin{align*}
\dim (\gra{X+f}{A}) = \min\{ \alpha/H, \alpha + d(1-H)\} \ \text{ and }
\ \dim (\ima{X+f}{A}) = \min\{\alpha/H,d\}.
\end{align*}
\end{theorem}

\begin{remark}\label{rem:secondterm}\rm{
Note that when $d=1$ and $A=[0,1]$, then the minimum in the expressions above is always the second term.
}
\end{remark}

We prove Therem~\ref{thm:dim-graph-image} in Section~\ref{sec:imgr}.
We now define a class of self-affine sets analysed by Bedford~\cite{Bedford_thesis} and  McMullen~\cite{McMullen}.

\begin{definition}\rm{
Let $n> m$ be two positive integers and $D\subseteq \{0,\ldots, n-1\}\times \{0,\ldots,m-1\}$. We call~$D$ a pattern. The self-affine set corresponding to the pattern~$D$ is defined to be
\[
K(D) = \left\{\sum_{k=1}^{\infty} (a_kn^{-k}, b_k m^{-k}): (a_k,b_k) \in D \ \text{ for all } \ k\right\}.
\]
We set $r(j) = \sum_{\ell=0}^{m-1} \1((j,\ell) \in D)$ for the number of rectangles or row~$j$.
}
\end{definition}

\begin{corollary}\label{cor:selfaffine}
Let $X$ be a fractional Brownian motion in~$\R$ of Hurst index~$H$.
Let $D\subseteq \{0,\ldots,n-1\}\times\{0,\ldots,m-1\}$ be a pattern such that there exists $f:[0,1]\to[0,1]$ with $\gr{f} = K(D)$ and $\log_n(m) < H$. Then almost surely
\[
\dim (\gr{X+f}) = 1-H + H\log_m\left( \sum_{j=0}^{m-1} r(j)^{\log_n (m)/H} \right).
\]
\end{corollary}

\begin{figure}[h!]
\subfigure{
\includegraphics[scale=0.45]{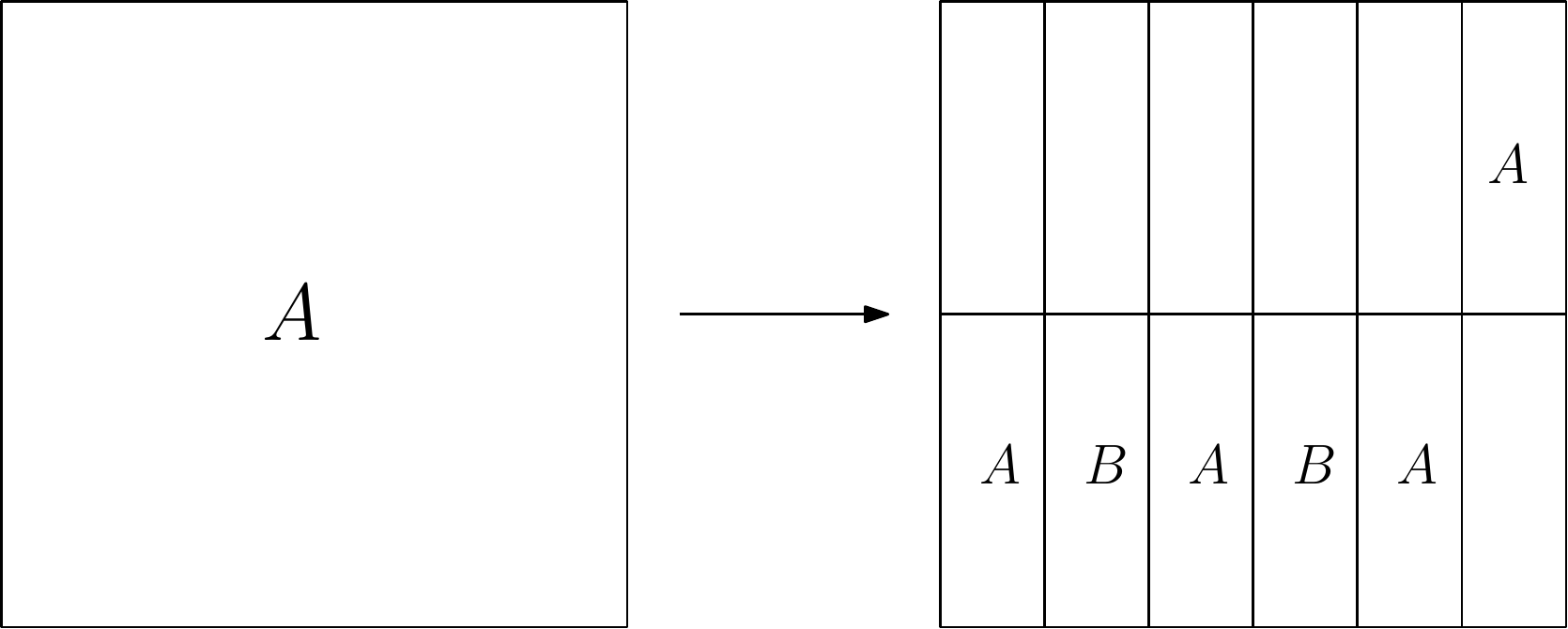}
}
\hspace{0.5cm}
\subfigure{
\includegraphics[scale=0.45]{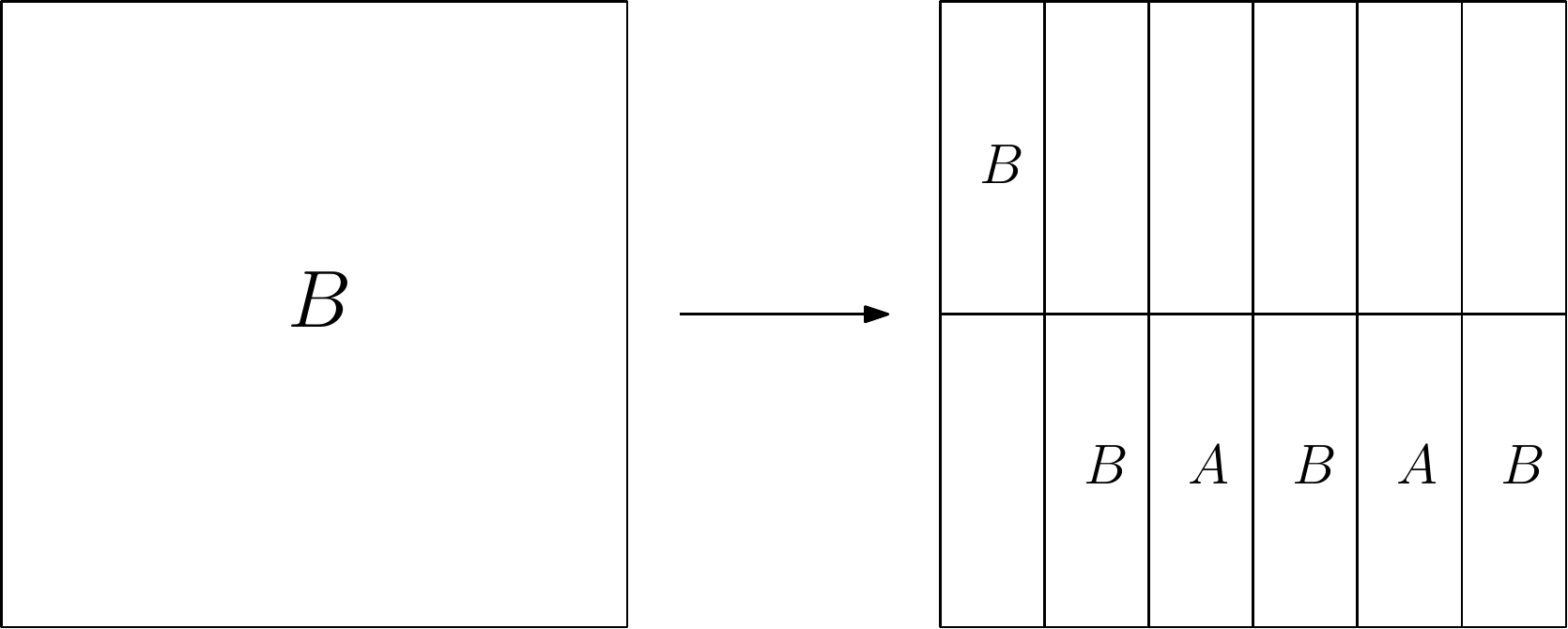}
}
\caption{\label{fig:patterns}The patterns $A$ and $B$ used in each iteration}
\end{figure}

In~\cite[Theorems~1.8 and~1.9]{PS10} it was shown that if~$B$ is a standard Brownian motion and~$f:[0,1]\to\R^d$ for~$d\geq 1$ is a continuous function, then almost surely
\[
\dim(\gr{B+f}) \geq \max\{\dim(\gr{f}),\dim(\gr{B})\} \text{ and } \dim(\im{B+f}) \geq \max\{\dim(\im{f}),\dim(\im{B})\}.
\]
In the same paper it was shown that in dimensions~$3$ and above there exist continuous functions~$f$ such that the Hausdorff dimension of the image and the graph are strictly larger than the maxima given above. In dimension~$1$ though, the question of finding a continuous function~$f$ with $\dim(\gr{B+f}) >\dim(\gr{f})$ remained open.

As an application of Theorem~\ref{thm:dim-graph-image} for the case of the graph we give an example of a function~$f$ which is H\"older continuous with parameter $\log 2/\log 6<1/2$ and for which we can calculate exactly the parabolic Hausdorff dimension.
The patterns used in each iteration of the construction of the graph of~$f$ are depicted in Figure~\ref{fig:patterns} and the first few approximations to the graph of $f$ are shown in Figure~\ref{fig:K}. We defer the formal definition to Section~\ref{sec:self-affine} where we also calculate the parabolic dimension of the graph of~$f$.

\begin{figure}[h!]
\begin{center}
\subfigure{
\includegraphics[scale=0.3]{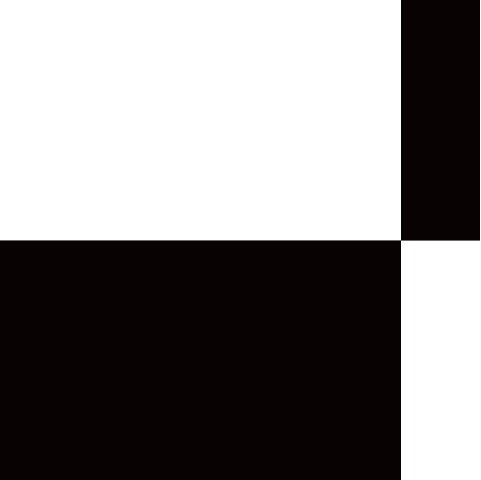}}
\subfigure{
\includegraphics[scale=0.3]{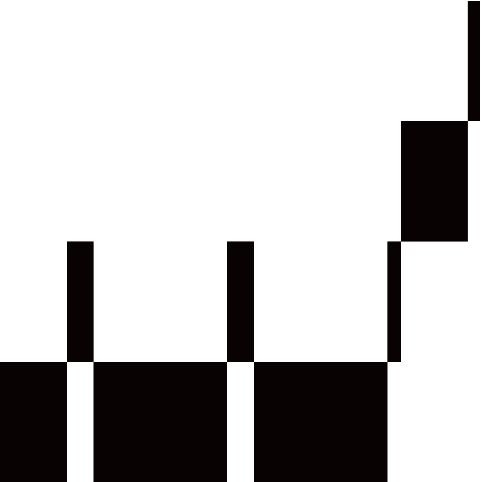}}
\subfigure{
\includegraphics[scale=0.3]{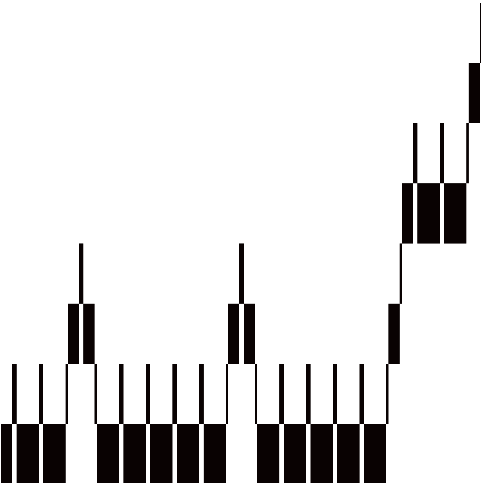}}
\subfigure{
\includegraphics[scale=0.3]{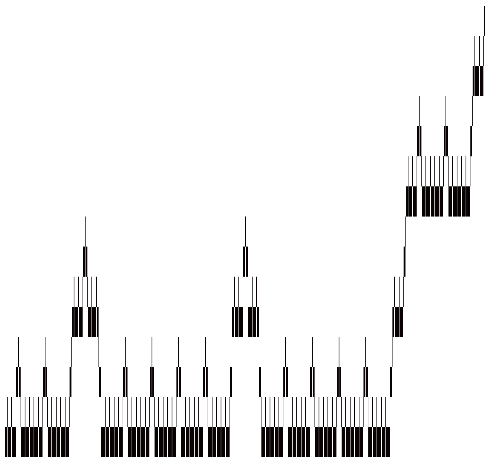}}
\caption{\label{fig:K}Finite approximations of $\gr{f}$}
\end{center}
\end{figure}

\begin{corollary}\label{cor:strictineq}

Let $B$ be a standard Brownian motion in one dimension. Then there exists a function $f:[0,1]\to [0,1]$ (the first approximations to its graph are depicted in Figure~\ref{fig:K}) which is H\"older continuous of parameter $\theta = \log 2/\log 6$, its graph is a self-affine set with $\dmin{\gr{f}}=1+\log 3/
\log 6$ and it satisfies almost surely
\[
\dim (\gr{B+f}) = \frac{1+ \log_2(5^{2\theta} + 1)}{2} > \max\left\{\dim(\gr{f}), \frac{3}{2} \right\} = \log_2(5^\theta + 1).
\]
\end{corollary}

\begin{figure}[h!]
\begin{center}
\includegraphics[scale=0.6]{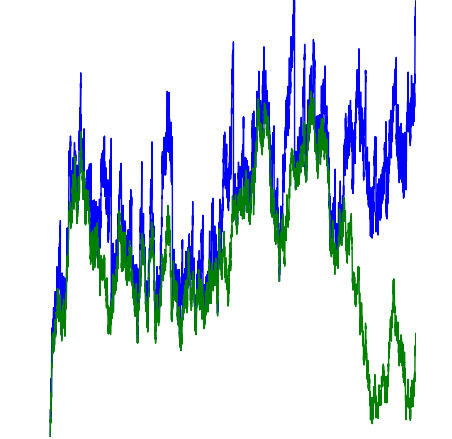}
\caption{\label{fig:B+f} The graph of $B+f$ in green and the graph of $B$ in blue with $f$ of Corollary~\ref{cor:strictineq}}
\end{center}
\end{figure}

We prove Corollaries~\ref{cor:selfaffine} and~\ref{cor:strictineq} in Section~\ref{sec:self-affine}.

\begin{remark}
\rm{
If $K(D)$ is a self-affine set corresponding to the pattern $D$ and $r(j)\geq 1$ for all~$j$, then McMullen~\cite{McMullen} showed
\begin{align}\label{eq:self-ineq}
\dim_M({K(D)}) = 1 + \log_n\frac{|D|}{m} \quad \text{and} \quad \dim(K(D))  = \log_m\left( \sum_{j=1}^{m} r(j)^{\log_n m} \right).
\end{align}
From~\cite[Theorem~1.8]{CPS12} we have that if $h:[0,1]\to \R$ is a continuous function, then almost surely
\begin{align}\label{eq:mineq}
\dim_M(\gr{B+h}) = \max\{ \dim_M(\gr{h}), \dim_M(\gr{B})\}.
\end{align}
The proof of \eqref{eq:self-ineq} applies to the graph of~$f$, where~$f$ is the function of Corollary~\ref{cor:strictineq}. In conjunction with \eqref{eq:mineq}, this gives that almost surely
\[
\dim_M({\gr{B+f}}) = \dim_M({\gr{f}}) = 1 + \log_6 3 = 1.6131... >\dim(\gr{B+f}) =  1.5807...
\]
This shows that despite the Brownian perturbations, the Hausdorff and Minkowski dimensions still disagree (as is the case for the graph of~$f$). Comparisons of Hausdorff and Minkowski dimensions for other self-affine graphs perturbed by Brownian motion are in Section~\ref{sec:mink}.
}
\end{remark}

{\bf Related work}
Khoshnevisan and Xiao~\cite{KhosXiao} also employ the parabolic dimension of Taylor and Watson~\cite{TaylorWatson} to determine the Hausdorff dimension of the image of Brownian motion intersected with a compact set. The problem of estimating the dimension of fractional Brownian motion with drift was studied by Bayart and Heurteaux~\cite{BayHeu} (the case of Brownian motion was considered in~\cite{PS10}). These papers obtain upper and lower bounds for the dimension which differ in general. The lower bounds are proved by the energy method. The novel aspect of Theorem~\ref{thm:dim-graph-image} is that it gives an exact expression for the dimension of the graph of $X+f$ valid for all Borel functions~$f$.

\section{Dimension of~$\gr{X+f}$ and~$\im{X+f}$}\label{sec:imgr}

In this section we prove Theorem~\ref{thm:dim-graph-image}. We start with an easy preliminary lemma that relates the parabolic Hausdorff dimension to Hausdorff dimension. 

Note that for functions $f,g$ we write $f(n)\lesssim g(n)$ if there exists a constant $c>0$ such that $f(n)\leq cg(n)$ for all $n$. We write $f(n)\gtrsim g(n)$ if $g(n)\lesssim f(n)$.

\begin{lemma}\label{lem:dim-dimpsi-general}
For all~$A\subseteq \R_+\times \R^d$ we have
\[
\dim(A) \leq \left(\dpsi{A} + d(1-H) \right) \wedge  \frac{\dpsi{A}}{H}.
\]
\end{lemma}

\begin{proof}[{\textbf{Proof}}]

For $\beta>0$ we let $\H_\beta(A)$ be the~$\beta$-Hausdorff content of $A$, i.e.
\[
\H_\beta(A) = \inf\left\{\sum_j |E_j|^\beta: A\subseteq \cup_j E_j\right\}.
\]
Let $\epsilon>0$ and $\eta<1$.
We set $\beta = \frac{\dpsi{A}}{H}+\frac{\epsilon}{H}$ and $\gamma -d(1-H)= \dpsi{A} +\epsilon$. Then $\Psi_{\beta H}(A) =0$, and hence there exists a cover $([a_j,a_j+\delta_j] \times [b_{j,1},b_{j,1}+\delta_j^H]\times\ldots \times [b_{j,d},b_{j,d}+\delta_j^H])_j$ of the set~$A$, such that
\begin{align}\label{eq:smalldelta}
\sum_{j} \delta_j^{\beta H} < \eta.
\end{align}
From~\eqref{eq:smalldelta} it follows that $\delta_j<1$ for all~$j$, and hence the diameter of every set in the above cover of~$A$ is at most $\sqrt{d} \delta_j^H$. Therefore we obtain
\begin{align}\label{eq:firstdim}
\H_\beta(A)\leq \sum_{j} d^{\beta/2}({\delta_j}^H)^\beta = d^{\beta/2}\sum_j \delta_{j}^{\beta H} <d^{\beta/2}\eta,
\end{align}
where in the last step we used~\eqref{eq:smalldelta}.
Each interval~$[b_{j,i},b_{j,i}+\delta_j^H]$ can be divided into $\delta_j^{H-1}$ intervals of length $\delta_j$ each.
(We omit integer parts to lighten the notation.) In this way we obtain a new cover of the set~$A$ which satisfies
\begin{align}\label{eq:seconddim}
\H_\gamma(A) \leq \sum_j \delta_j^{(H-1)d} \delta_j^\gamma = \sum_j \delta_j^{\beta H}<\eta.
\end{align}
From~\eqref{eq:firstdim} and~\eqref{eq:seconddim} we deduce that
\[
\dim(A) \leq \beta \wedge \gamma = \left(\frac{\dpsi{A}}{H} + \frac{\epsilon}{H}\right) \wedge (\dpsi{A} + d(1-H) + \epsilon).
\]
Therefore letting $\epsilon$ go to $0$ we conclude
\[
\dim(A) \leq \left(\dpsi{A} + d(1-H) \right)\wedge \frac{\dpsi{A} }{H}
\]
and this finishes the proof.
\end{proof}

\begin{lemma}\label{lem:equaldpsi}
Let $f:[0,1]\to \R^d$ be a Borel measurable function. Then for all Borel sets $A\subseteq [0,1]$ almost surely
\[
\dpsi{\gra{X+f}{A}} = \dpsi{\gra{f}{A}}.
\]
\end{lemma}

\begin{proof}[\textbf{Proof}]

Since $X$ is a fractional Brownian motion of Hurst index~$H$, it follows that it is almost surely H\"older continuous of parameter $H-\epsilon$ for all $\epsilon>0$ (see for instance~\cite[Section~18]{Kahane}).
Therefore, for~$\zeta>0$ there exists a constant $C$ such that almost surely for all $s,t\in [0,1]$ we have
\begin{align}\label{eq:holder}
\norm{X_s- X_t} \leq C |t-s|^{H-\zeta}.
\end{align}
Let $0<\eta<h_0$. We set $\alpha = \dpsi{\gra{f}{A}}$. Then $\Psi_{\alpha+\epsilon}(\gra{f}{A})=0$, and hence there exists a cover~$([a_j,a_j+\delta_j]\times [b_j^1,b_{j,1}+\delta_j^H]\times \ldots \times [b_{j,d},b_{j,d}+\delta_j^H])_j$ of~$\gra{f}{A}$ such that
\begin{align}\label{eq:sumdelta}
\sum_j \delta_j^{\alpha+\epsilon} <\eta.
\end{align}
Using this cover we will derive a cover of~$\gra{X+f}{A}$. By~\eqref{eq:holder} if $t\in [a_j,a_j+\delta_j]$, then
\begin{align*}
\|X_t - X_{a_j}\| \leq C \delta_j^{H-\zeta}.
\end{align*}
Therefore the collection of sets
\[
\left(\left[a_j,a_j+9C^2\delta_j^{1-\zeta/H} \right] \times \left[r_{j,1}, r_{j,1} +3C\delta_j^{H-\zeta}\right]\times \ldots \times \left[r_{j,d}, r_{j,d} + 3C\delta_j^{H-\zeta}\right]\right)_j,
\]
where $r_{j,i}=b_{j,i}+X^i_{a_j} -C\delta_j^{H-\zeta}$ is a cover of
$\gra{X+f}{A}$. From~\eqref{eq:sumdelta} we obtain that for a positive constant $c$ we have
\[
\Psi_{\alpha+2\epsilon}(\gra{X+f}{A}) \leq (9C^2)^{\alpha+2\epsilon}\sum_{j} \left(\delta_j^{1-\zeta/H} \right)^{\alpha+2\epsilon} \leq c\sum_j \delta_j^{\alpha+\epsilon} <c \eta,
\]
where the penultimate inequality follows by choosing~$\zeta>0$ sufficiently small and $c$ is a positive constant.
We thus showed that almost surely $\Psi_{\alpha+2\epsilon}(\gra{X+f}{A}) =0$ for all $\epsilon>0$, which implies that almost surely
\[
\dpsi{\gra{X+f}{A}} \leq \alpha.
\]
The other inequality follows in the same way and this concludes the proof.
\end{proof}

Lemmas~\ref{lem:dim-dimpsi-general} and~\ref{lem:equaldpsi} give the following:

\begin{corollary}\label{cor:parabdim}
Let $f:[0,1]\to \R^d$ be a function. Then
almost surely we have
\[
\dim(\gra{X+f}{A}) \leq \frac{\dpsi{\gra{f}{A}}}{H} \wedge \left(\dpsi{\gra{f}{A}} + d(1-H)\right).
\]
\end{corollary}

We now recall the definition of the capacity of a set.

\begin{definition}\rm{
Let $K:\R^d \to [0,\infty]$ and $A$ a Borel set in~$\R^d$. (Sometimes $K$ is called a \it{difference kernel}.) The $K$-energy of a measure~$\mu$ is defined to be
\[
\EE_K(\mu) = \int \int K(x-y) \,d\mu(x)d\mu(y)
\]
and the $K$-capacity of $A$ is defined as
\[
\cpc{K}{A} = [\inf\{\EE_K(\mu): \mu \text{ a probability measure on } A\}]^{-1}.
\]
When the kernel has the form $K(x) = |x|^{-\alpha}$, then we write
$\mathcal{E}_\alpha(\mu)$ for $\mathcal{E}_K(\mu)$ and
$\mathrm{Cap}_\alpha(A)$
for $\mathrm{Cap}_K(A)$ and we refer to them as the $\alpha$-energy of $\mu$ and the Riesz $\alpha$-capacity
of $A$ respectively.
}
\end{definition}

We recall the following theorem which gives the connection between the Hausdorff dimension of a set and its Riesz $\alpha$-capacity. For the proof see~\cite{Carlesson}.

\begin{theorem}[Frostman]\label{thm:dimcap}
For any Souslin set $A \subset \R^d$,
\[
\dim(A) = \sup \{\alpha: \mathrm{Cap}_\alpha(A) >0 \}.
\]
\end{theorem}

Let $X$ be a fractional Brownian motion in~$\R^d$ of Hurst index~$H$. For $(s,x) \in \R_+ \times \R^d$ we define the difference kernel
\begin{align}\label{eq:defigamma}
 \ig(s,x) = \E{\frac{1}{(\|X_s+x\|^2 + s^2)^{\gamma/2}}}.
\end{align}


\begin{lemma}\label{lem:capdim}
Let $X$ be a fractional Brownian motion in $\R^d$ of Hurst index~$H$ and let $f:[0,1]\to \R^d$ be a Borel measurable function. Let~$A$ be a closed subset of $[0,1]$.
If $\cpc{\ig}{\gra{f}{A}}>0$, then almost surely
\[
\cpc{\gamma}{\gra{X+f}{A}}>0.
\]
\end{lemma}

\begin{proof}[\textbf{Proof}]

Since by assumption $\cpc{\ig}{\gra{f}{A}}>0$, there exists a probability measure $\nu_f$ on $\gra{f}{A}$ with finite energy, i.e.\
\begin{align*}
\int_{\gra{f}{A}} \int_{\gra{f}{A}} \ig(s-t,f(s)-f(t)) \,d\nu_f(s,f(s)) \,d\nu_f(t,f(t)) \\=  \int_A \int_A \ig(s-t,f(s)-f(t)) \,d\nu(s) \,d\nu(t) <\infty,
\end{align*}
where $\nu$ is the measure on $A$ satisfying
$\nu = \nu_f \circ H^{-1}$, where $H((s,f(s))) = s$ is the projection mapping.
We now define a measure $\til{\nu}$ on $\gra{X+f}{A}$ via
\[
\til{\nu}(A) = \nu(\{t: (t,(X+f)(t)) \in A\}).
\]
We will show that this measure has finite $\gamma$ energy. Indeed,
\begin{align*}
\E{\int \int \frac{1}{\norm{x-y}^{\gamma}} \,d\til{\nu}(x) \,d\til{\nu}(y)} &= \E{\int \int \frac{d\nu(s) d\nu(t)}{\left(\norm{(X+f)(t) - (X+f)(s)}^2 + |t-s|^2 \right)^{\gamma/2}}} \\
&= \int \int \ig(s-t,f(s)-f(t)) \,d\nu(s) \,d\nu(t) <\infty,
\end{align*}
and hence it follows that $\cp_{\gamma}(\gra{X+f}{A})>0$ almost surely.
\end{proof}

\begin{lemma}\label{lem:capacity}
Let~$f:[0,1]\to\R^d$ be a bounded Borel measurable function and~$A$ a closed subset of~$[0,1]$. If~$\alpha = \dpsi{\gra{f}{A}}$, then
\[
\min\left\{ \frac{\alpha}{H}, \alpha + d(1-H)\right\} \leq \inf\{\gamma: \ \cp_{\ig}(\gra{f}{A}) =0\}.
\]
\end{lemma}

Before proving Lemma~\ref{lem:capacity} we show how we can bound from above the kernel $\ig$ in three different regimes.

\begin{lemma}\label{lem:calculations}

Fix $M>0$. There exists a positive constant $C$ such that for all $t\in (0,1/e]$ and all~$u$ satisfying $\|u\|\leq M$, the kernel $\ig$ defined in~\eqref{eq:defigamma} satisfies
\begin{align*}
\ig(t,u) \lesssim
\begin{cases} \norm{u}^{-\gamma} & \text{if $\norm{u} >C t^{H}\sqrt{|\log t|}$,}
\\
t^{d(1-H)-\gamma} &\text{if $\norm{u}\leq C t^H\sqrt{|\log t|}$ and $d<\gamma$,}
\\
t^{-\gamma H} &\text{if $\norm{u}\leq Ct^H\sqrt{|\log t|}$ and $d>\gamma$}.
\end{cases}
\end{align*}
\end{lemma}

\begin{proof}[\textbf{Proof}]
By scaling invariance of fractional Brownian motion we have
\[
\ig(t,u) = \E{\frac{1}{\left( \norm{t^{H} X_1 + u}^2 + t^2    \right)^{\gamma/2}}}.
\]
Let $C$ be a constant to be determined and let $\norm{u} > C t^H\sqrt{ |\log t|}$. By the Gaussian tail estimate we have
\[
\pr{t^H\norm{X_1} >\frac{\norm{u}}{2}} \leq 2 e^{-\frac{\norm{u}^2}{8t^{2H}}} \leq 2 e^{-\frac{C^2\log(1/t)}{8}} \leq 2 t^{C^2/8}.
\]
On the event $\{t^H \norm{X_1} < \norm{u}/2\}$ we have
\[
\norm{t^H X_1 + u} \geq \norm{u} - t^H \norm{X_1} \geq \frac{\norm{u}}{2}.
\]
Therefore, taking $C$ sufficiently large we get
\begin{align*}
\E{\frac{1}{\left( \norm{t^H X_1 + u}^2 + t^2    \right)^{\gamma/2}}} &\lesssim \pr{t^H\norm{X_1} >\frac{\norm{u}}{2}} \frac{1}{t^{\gamma}} + \pr{t^H\norm{X_1} \leq \frac{\norm{u}}{2}} \frac{1}{\norm{u}^\gamma} \\
&\lesssim  t^{C^2/8 -\gamma} + \norm{u}^{-\gamma} \lesssim
\norm{u}^{C^2/8 - \gamma} +
\norm{u}^{-\gamma} \lesssim \norm{u}^{-\gamma},
\end{align*}
since $\norm{u}\leq M$ and this finishes the proof of the first part.
Next, let $\norm{u} \leq  C t^H\sqrt{ |\log t|}$.
Then
\begin{align*}
\E{\frac{1}{\left(\norm{t^H X_1 + u}^2 + t^2 \right)^{\gamma/2}}} = \int \frac{1}{\left(\norm{x+u}^2 + t^2 \right)^{\gamma/2}} e^{-\frac{\norm{x}^2}{2t^{2H}}}\,dx = \int f(x+u) g(x)\,dx,
\end{align*}
where $f(x)  = \left(\norm{x}^2+t^2 \right)^{-\gamma/2}$ and $g(x) = e^{-\norm{x}^2/(2t^{2H})}$. Since they are both decreasing as functions of $\norm{x}$, it follows that
\begin{align}\label{eq:fkg}
\int (f(x+u) - f(x)) (g(x+u) - g(x)) \,dx \geq 0,
\end{align}
and hence this gives
\begin{align*}
\E{\frac{1}{\left(\norm{t^HX_1 + u}^2 + t^2 \right)^{\gamma/2}}} &\leq \E{\frac{1}{\left(\norm{t^H X_1}^2 + t^2 \right)^{\gamma/2}}}
\leq \frac{1}{t^\gamma} \pr{\norm{t^H X_1} \leq t}
+ \frac{1}{t^{\gamma H}}\E{\frac{\1(\norm{t^HX_1}>t)}{\norm{X_1}^\gamma}} \\
&= \frac{1}{t^\gamma}  \int_{\B(0,t^{1-H})} \frac{1}{(2\pi)^{d/2}} e^{-\norm{x}^2/2} \,dx+ \frac{1}{t^{\gamma H}} \int_{\B(0,t^{1-H})^c}  \frac{1}{(2\pi)^{d/2}\norm{x}^{\gamma}} e^{-\norm{x}^2/2} \,dx \\
&\lesssim t^{(1-H)d-\gamma} + t^{-\gamma H} \int_{t^{1-H}}^{\infty} r^{d-1-\gamma} e^{-r^2/2}\,dr
\\&\lesssim t^{(1-H)d-\gamma} + t^{-\gamma H} \int_{t^{1-H}}^{1} r^{d-1-\gamma} \,dr + c_1,
\end{align*}
where $c_1$ is a positive constant. If $d>\gamma$, then from the above we deduce that
\[
\ig(t,u) \lesssim t^{-\gamma H},
\]
while when $d<\gamma$, then
\[
\ig(t,u) \lesssim t^{d(1-H)-\gamma}
\]
and this concludes the proof of the lemma.
\end{proof}

The next theorem is the analogue of Frostman's theorem for parabolic Hausdorff dimension.The statement can be found in Taylor and Watson~\cite[Lemma~4]{TaylorWatson} and the proof follows along the same lines as the proof of Frostman's theorem for Hausdorff dimension. We include the statement here for the reader's convenience.

\begin{theorem}[Frostman's theorem]\label{thm:Frostman} Let $A$ be a Borel set.
If $\dpsi{A} >\beta$, then there exists a Borel probability measure $\mu$ supported on $A$ such that
\[
\mu([a,a+\delta] \times \cup_{j}[b_{j,1},b_{j,1}+ \delta^H] \times \ldots [b_{j,d},b_{j,d} + \delta^H]) \leq C\delta^\beta,
\]
where $C$ is a positive constant.
\end{theorem}

We now give the proof of Lemma~\ref{lem:capacity}.

\begin{proof}[\textbf{Proof of Lemma~\ref{lem:capacity}}]

Let $\beta = \alpha - \epsilon/2$.
Since the graph of a Borel function is always a Borel set, it follows by Theorem~\ref{thm:Frostman} that there exists a probability measure $\mu$ supported on $\gra{f}{A}$ such that
\begin{align}\label{eq:frostman}
\mu([a,a+\delta] \times \cup_{j=1}^{d} [b_j,b_j + \delta^H]) \leq c_2 \delta^\beta.
\end{align}
From this it follows that the measure~$\mu$ is non-atomic.
Suppose first that $\min\{\alpha/H, \alpha+d(1-H)\} = \alpha/H$.
Let $\gamma = \beta/H - \epsilon<d$. We show that $\cpc{\ig}{\gra{f}{A}}>0$. It suffices to prove that
\begin{align}\label{eq:goalenergy}
\EE_{\ig}(\mu) < \infty.
\end{align}
Since $\gamma<d$ and $f$ is bounded on $[0,1]$, if we define
\begin{align*}
I_1 & = \iint\limits_{|s-t|<1/e}  \1(\norm{f(t) - f(s) } \leq C |t-s|^H\sqrt{|\log|t-s||}) |t-s|^{-\gamma H} \,d\mu((s,f(s))) d\mu((t,f(t))), \\
I_2 &=\iint\limits_{|s-t|<1/e} \1(\norm{f(t) - f(s) } > C |t-s|^H\sqrt{|\log|t-s||}) \norm{f(t) - f(s)}^{-\gamma} \,d\mu((s,f(s))) d\mu((t,f(t))),
\end{align*}
then from Lemma~\ref{lem:calculations} we get that
\begin{align*}
\EE_{\ig}(\mu) &= \iint \E{\frac{1}{( \norm{X_s - X_t + f(s) - f(t)}^2 + |t-s|^2  )^{\gamma/2}}} \,d\mu((s,f(s))) d\mu((t,f(t))) \lesssim
e^{\gamma} + I_1 + I_2.
\end{align*}
We first show that $I_1<\infty$. Since $\mu$ is non-atomic, we have
\begin{align}\label{eq:i1first}
I_1 \leq \sum_{k=0}^{\infty} 2^{k\gamma H} \mu\otimes \mu\{2^{-k} \leq |t-s| < 2^{-k+1}, \norm{f(t) - f(s)} \leq C2^{-k H}\sqrt{k} \}.
\end{align}
Let $M = \max_{t\in[0,1]} \norm{f(t)}<\infty$. Then the measure~$\mu$ is supported on~$[0,1]\times[-M,M]^d$. For $k>0$, we partition the space $[0,1]\times[-M,M]^d$ into rectangles of of dimensions $2^{-k}\times 2^{-k H}\times\ldots\times 2^{-k H}$. We let $\D_k$ be the collection of rectangles of generation~$k$. For two rectangles $Q,Q'$ of the same generation we write $Q\sim Q'$ if
there exist $(s,x) \in Q, (t,y) \in Q'$ such that $2^{-k} \leq |s-t| < 2^{-k+1}$ and $\norm{x-y} \leq C2^{-k H} \sqrt{k}$. Then from~\eqref{eq:i1first} we obtain
\begin{align*}
I_1 \leq \sum_{k=0}^{\infty} 2^{k\gamma H} \sum_{\substack{Q,Q' \in \D_k \\ Q \sim Q'}} \mu\otimes \mu (Q\times Q') =  \sum_{k=0}^{\infty} 2^{k\gamma H} \sum_{\substack{Q,Q' \in \D_k \\ Q \sim Q'}} \mu(Q) \mu(Q').
\end{align*}
We now notice that if we fix $Q\in \D_k$, then the number of $Q'$ such that $Q\sim Q'$ is up to constants~$k^{d/2}$. Using the obvious inequality
\begin{align}\label{eq:easyineq}
\mu(Q) \mu(Q') \leq \frac{1}{2}(\mu(Q)^2 + \mu(Q')^2)
\end{align}
and~\eqref{eq:frostman} to get~$\mu(Q)\leq c_2 2^{-k\beta}$ we deduce
\begin{align*}
I_1 \lesssim \sum_{k=0}^{\infty} 2^{k\gamma H} k^{d/2} \sum_{Q \in \D_k} \mu(Q)^2 \lesssim  \sum_{k=0}^{\infty} 2^{k\gamma H} k^{d/2} 2^{-k\beta}\sum_{Q \in \D_k} \mu(Q) = \sum_{k=0}^{\infty} \frac{k^{d/2}}{2^{k\epsilon H}}<\infty,
\end{align*}
since $\sum_{Q\in \D_k} \mu(Q)=1$ as $\mu$ is a probability measure. It remains to show that $I_2<\infty$. By defining a new equivalence relation on rectangles in $\D_k$, i.e.\ that $Q\sim Q'$ if there exist $(s,x)\in Q, (t,y)\in Q'$ such that $|t-s|\leq 2^{-k}$ and $2^{-k H}\leq \norm{f(t)-f(s)} <2^{-kH  + H}$ we get
\begin{align*}
I_2 \lesssim \sum_{k=0}^{\infty} 2^{k\gamma H}\mu\otimes \mu\{2^{-k H}\leq \norm{f(t)- f(s)} \lesssim 2^{-k H + H}, |t-s| \leq 2^{-k}\} \lesssim \sum_{k=0}^{\infty} 2^{k\gamma H} 2^{-k\beta} <\infty,
\end{align*}
where we used~\eqref{eq:easyineq} again and the fact that the number of $Q'\in \D_k$ such that $Q\sim Q'$ is of order~$1$. This completes the proof in the case when $\alpha/H<d$. Suppose now that $\alpha/H> d$.
Take $\epsilon>0$ small enough such that $\alpha-2\epsilon>dH$ and set $\beta = \alpha - \epsilon$. Let $\gamma = \beta + d(1-H) - \epsilon$.
Then using the measure~$\mu$ from~\eqref{eq:frostman} and following the same steps as above we can write the same expression for the energy. Then, since $\gamma>d$, the quantity $I_1$ in view of  Lemma~\ref{lem:calculations} is bounded by
\[
I_1 \lesssim \int \int  \1(\norm{f(t) - f(s) } \leq C |t-s|^H\sqrt{|\log|t-s||}) |t-s|^{d(1-H)-\gamma} \,d\mu((s,f(s))) d\mu((t,f(t))).
\]
Following the same steps as earlier we deduce
\begin{align*}
I_1 \lesssim \sum_{k=0}^{\infty} 2^{-k(d(1-H) - \gamma)} k^{d/2} 2^{-k\beta} = \sum_{k=0}^{\infty} 2^{-k\epsilon} k^{d/2}<\infty.
\end{align*}
For the quantity~$I_2$ in the same was as above we have
\[
I_2\lesssim \sum_{k=0}^{\infty} 2^{k\gamma H} 2^{-k\beta} = \sum_{k=0}^{\infty} 2^{-k((1-H)(\alpha - d H) +2\epsilon H -\epsilon )}<\infty,
\]
since $\alpha-2\epsilon>dH$ and this completes the proof of the lemma.
\end{proof}

\begin{claim}\label{cl:infsup}
Let $A\subseteq \R\times\R^d$. Then
\[
\inf\{\gamma: \cpc{\ig}{A} =0\} = \sup\{\gamma: \cpc{\ig}{A}>0\}.
\]
\end{claim}

%

\begin{proof}[\textbf{Proof of Theorem~\ref{thm:dim-graph-image}}]
(dimension of the graph)

We first assume that $f$ is bounded.
We set $\alpha=\dpsi{\gra{f}{A}}$.
In view of Corollary~\ref{cor:parabdim} we only need to show that almost surely
\begin{align}\label{eq:goaldim}
\dim(\gra{X+f}{A}) \geq \alpha/H \wedge \left(\alpha + d(1-H) \right).
\end{align}
Claim~\ref{cl:infsup} gives that
\[
\inf\{\gamma: \ \cpc{\ig}{\gra{f}{A}} =0\} = \sup\{\gamma: \ \cpc{\ig}{\gra{f}{A}} >0\} = \gamma_*.
\]
Let $\gamma_n$ be such that $\cpc{I_{\gamma_n,H}}{\gra{f}{A}}>0$ and $\gamma_n\to \gamma_*$ as~$n\to \infty$. Then by Lemma~\ref{lem:capdim} we get that for all~$n$ a.s.\ $\cpc{\gamma_n}{\gra{X+f}{A}}>0$, and hence a.s.\
\[
\dim(\gra{X+f}{A}) \geq \gamma_n \ \text{ for all } \ n,
\]
which gives that almost surely~$\dim(\gra{X+f}{A})\geq \gamma_*$.
This combined with Lemma~\ref{lem:capacity} implies that almost surely
\[
\dim(\gra{X+f}{A}) \geq \min\left\{\frac{\alpha}{H},\left(\alpha+d(1-H)\right)\right\}
\]
and this concludes the proof in the case when $f$ is bounded. For the general case, we define the increasing sequence of sets $A_n = \{s\in A: |f(s)| \leq n\}$. Then by the countable stability property of Hausdorff and parabolic dimension we have
\begin{align}\label{eq:increasing}
\dim(\gra{X+f}{A}) = \sup_{n} \dim(\gra{X+f}{A_n}) \quad \text{and} \quad \dpsi{\gra{f}{A_n}}\uparrow \dpsi{\gra{f}{A}}.
\end{align}
From above we have
\begin{align*}
\dim(\gra{X+f}{A_n})= \min\left\{ \frac{\dpsi{\gra{f}{A_n}}}{H} , \dpsi{\gra{f}{A_n}}+d(1-H)\right\}.
\end{align*}
Using this and~\eqref{eq:increasing} proves the theorem in the general case.
\end{proof}

\begin{proof}[{\textbf{Proof of Theorem~\ref{thm:dim-graph-image}}}]
(dimension of the image)

As in the proof of Theorem~\ref{thm:dim-graph-image} in the case of the graph, we can assume that $f$ is bounded. The general case follows exactly in the same way as for the graph.

The dimension of the image satisfies
\[
\dim (\ima{X+f}{A}) \leq \dim(\gra{X+f}{A}) \leq  \frac{\dpsi {\gra{X+f}{A}}}{H} = \frac{\dpsi{\gra{f}{A}}}{H} =\frac{\alpha}{H},
\]
where the second inequality follows from Lemma~\ref{lem:dim-dimpsi-general} and the first equality follows from Lemma~\ref{lem:equaldpsi}. Hence the upper bound on the dimension of~$\ima{X+f}{A}$ is immediate. It only remains to show the lower bound. Let $\beta = \alpha\wedge d H -\epsilon H$ and $\gamma = \beta/H - \epsilon$. Then since the image of a Borel set under a Borel measurable function is a Souslin set (see for instance~\cite{Kechris}), it follows from Theorem~\ref{thm:dimcap} that
 it suffices to show that $\cpc{\gamma}{\ima{X+f}{A}} >0$, i.e.\ it is enough to find a measure of finite $\gamma$-energy. By Theorem~\ref{thm:Frostman} there exists a probability measure $\mu$ on $\gra{f}{A}$ such that
\[
\mu([a,a+\delta] \times \cup_{j=1}^{d} [b_j, b_j + \delta^H) \leq \delta^\beta.
\]
Let $H$ be the projection mapping from $\gra{f}{A}$ to $A$, i.e.\ $H((s,f(s)))=s$ for all $s$. Let $\nu$ be the measure on~$A$ such that
\[
\nu = \mu \circ H^{-1}.
\]
Let $\til{\mu}$ be a measure on $\gra{X+f}{A}$ given by
\[
\til{\mu}(R) = \nu((X+f)^{-1}(R))
\]
where $R \subseteq \ima{X+f}{A}$.
We will show that almost surely
\[
\EE_\gamma(\til{\mu}) = \int \int \frac{d\til{\mu}(x) d\til{\mu}(y)}{\norm{x-y}^{\gamma}} <\infty.
\]
Taking expectations we get
\begin{align*}
\E{\EE_{\gamma}(\til{\mu})} = \int \int  \E{\frac{1}{\norm{X_s - X_t + f(s) - f(t)}^\gamma}}  d\mu((s,f(s))) d\mu((t,f(t))).
\end{align*}
We now show that
\begin{align}\label{eq:minprevious}
\E{\frac{1}{\norm{X_s - X_t + f(s) - f(t)}^\gamma}} \lesssim \min\{\norm{f(t) - f(s)}^{-\gamma}, |t-s|^{-\gamma H} \}.
\end{align}
The calculations that lead to~\eqref{eq:minprevious} can be found in the proof of~\cite[Theorem~1.8]{PS10}, but we include the details here for the convenience of the reader.
Using~\eqref{eq:fkg} we have
\begin{align*}
 \E{\frac{1}{\norm{X_s - X_t + f(s) - f(t)}^\gamma}} \leq \E{\frac{1}{\norm{X_s - X_t}^\gamma}} \lesssim |t-s|^{-\gamma H}.
\end{align*}
We set $u=(f(s)-f(t))/|s-t|^H$ and we get
\begin{align*}
\E{\frac{1}{\norm{X_s - X_t + f(s) - f(t)}^\gamma}} = \frac{1}{|t-s|^{\gamma H}} \int_{\R^d} \frac{1}{(2\pi)^{d/2} \norm{x+u}^\gamma} e^{-\norm{x}^2/2}\,dx.
\end{align*}
We now upper bound the last integral appearing above
\begin{align*}
\int_{\R^d} \frac{1}{\norm{x+u}^\gamma} e^{-\norm{x}^2/2}\,dx &= \int_{\norm{x+u}\geq \norm{u}/2} \frac{1}{\norm{x+u}^\gamma} e^{-\norm{x}^2/2}\,dx  + \int_{\norm{x+u}< \norm{u}/2} \frac{1}{\norm{x+u}^\gamma} e^{-\norm{x}^2/2}\,dx  \\
&\lesssim \frac{1}{\norm{u}^{\gamma}} + e^{-\norm{u}^2/4} \int_{\norm{x}<\norm{u}} \frac{1}{\norm{x}^{\gamma}} \,dx \lesssim \norm{u}^{-\gamma},
\end{align*}
where the last step follows from passing to polar coordinates and using the fact that $d>\gamma$. Therefore multiplying the last upper bound by $|t-s|^{-\gamma H}$  proves~\eqref{eq:minprevious}. We now  need to decompose the energy in these two regimes, i.e.\ for $\norm{f(t) - f(s)} \leq |t-s|^H$ and $\norm{f(t)-f(s)}>|t-s|^H$. This now follows in the same way as the proof that $I_1, I_2<\infty$ in the proof of Lemma~\ref{lem:capacity}.
\end{proof}

\section{Self-affine sets}\label{sec:self-affine}

In this section we give the proofs of Corollaries~\ref{cor:selfaffine} and~\ref{cor:strictineq}. We start by calculating the parabolic Hausdorff dimension of any self-affine set as defined in the Introduction. Then we use Theorem~\ref{thm:dim-graph-image} to prove Corollary~\ref{cor:selfaffine}.

\begin{lemma}\label{lem:self-affine}
Let~$n>m$ and let~$D\subseteq \{0,\ldots, n-1\} \times \{0,\ldots, m-1\}$ be a pattern. If $\log_n
(m)<H$, then
\[
\dpsi{K(D)} = H\log_m\left(\sum_{j=0}^{m-1}r(j)^{\log_n(m)/H} \right),
\]
where $r(j) = \sum_{\ell=0}^{m-1} \1((j,\ell) \in D)$.
\end{lemma}

Before proving this lemma, we state the analogue of Billingsley's lemma for the parabolic Hausdorff dimension. See Billingsley~\cite{Billingsley} and Cajar~\cite{Cajar} for the proof. We first introduce some notation. Let~$b$ be an integer. We define the $b$-adic rectangles contained in~$[0,1]^2$ of generation~$k$ to be
\[
R_{k} = \left[\frac{(j-1)}{[b^{k/H}]}, \frac{j}{[b^{k/H}]}\right) \times \left[\frac{(i-1)}{b^k}, \frac{i}{b^k}\right),
\]
where $j$ ranges from $1$ to $[b^{k/H}]$ and $i$ ranges from $1$ to $b^k$,
and we write~$R_k(x)$ for the unique dyadic rectangle containing~$x$.

\begin{lemma}[Billingsley's lemma]\label{lem:billingsley}
Let~$A$ be a Borel subset of~$[0,1]^2$ and let~$\mu$ be a measure on~$[0,1]^2$ with~$\mu(A)>0$. If for all~$x \in A$ we have
\[
\alpha \leq \liminf_{n\to \infty} \frac{\log( \mu(R_n(x)))}{\log(b^{-n/H})} \leq \beta,
\]
then $\alpha\leq \dpsi{A} \leq \beta$.
\end{lemma}

We are now ready to give the proof of Lemma~\ref{lem:self-affine}. The proof follows the steps for the calculation of the Hausdorff dimension of a self-affine set as given in~\cite{McMullen} and~\cite{Perescarpet}.

\begin{proof}[{\textbf{Proof of Lemma~\ref{lem:self-affine}}}]

For $(x,y) \in [0,1]^2$ we define $Q_k(x,y)$ to be the closure of the set of points~$x',y'$ such that the first $\lfloor \theta k/H\rfloor$ digits of $x'$ and $x$ agree in the $n$-ary expansion and the first $k$ digits of $y'$ and $y$ agree in the $m$-ary expansion. Let $\pi=(p(d), d\in D)$ be a probability measure on~$D$.

Let $\mu$ be the image of the product measure $\pi^{\otimes\N}$ under the map
\[
R: \{(a_k,b_k)\}_{k\geq 1} \mapsto \sum_{k= 0}^{\infty}\left( a_kn^{-k}, b_km^{-k} \right),
\]
where $(a_k,b_k)\in D$ for all $k$.
We now consider the rectangle $n^{-k}\times m^{-k}$ defined by specifying the first $k$ digits of the base $n$ expansion of $x$ and the first $k$ digits of the base $m$ expansion of $y$. This has $\mu$ measure equal to $\prod_{i=1}^{k} p(x_i,y_i)$. Since $r(j)$ is the number of rectangles contained in row~$j$ of the pattern, it follows that the rectangle $Q_{k}(x,y)$ contains $\prod_{i=\lfloor \theta k/H\rfloor + 1}^{k} r(y_i)$ rectangles of size~$n^{-k}\times m^{-k}$.
We now assume that $p(d)$ only depends on the second coordinate. Hence we get
\begin{align}\label{eq:musquare}
\mu(Q_k(x,y)) = \prod_{\ell=1}^{k} p(x_\ell,y_\ell) \prod_{\ell=\lfloor \theta k/H \rfloor+1}^{k} r(y_\ell).
\end{align}
Taking logarithms of~\eqref{eq:musquare} we obtain
\begin{align}\label{eq:logmu}
\log(\mu(Q_k(x,y))) = \sum_{\ell=1}^{k} \log(p(x_\ell,y_\ell)) + \sum_{\ell=\lfloor \theta k/H\rfloor +1}^{k} \log(r(y_\ell)).
\end{align}
Since the digits $(x_\ell,y_\ell)_\ell$ are i.i.d.\ wrt to the product measure~$\pi^{\otimes\N}$, by the strong law of large numbers we get
\[
\lim_{k\to \infty} \frac{1}{k} \log(\mu(Q_k(x,y))  ) = \sum_{d\in D} p(d) \log(p(d)) + (1-\theta/H) \sum_{d\in D} p(d) \log(r(d))
\]
for $\mu$-almost every $x,y$.

Let~$A$ be the set of $(x,y)$ for which the convergence holds. Then $\mu(A^c)=0$. By the definition of the measure~$\mu$ it is clear that it is supported on the set~$K(D)$. Hence $\mu(K(D)^c\cup A^c) =0$ and for all~$x\in K(D)\cap A$ we have
\[
\lim_{k\to \infty} \frac{1}{k} \log(\mu(Q_k(x,y))  ) = \sum_{d\in D} p(d) \log(p(d)) + (1-\theta/H) \sum_{d\in D} p(d) \log(r(d)).
\]
Therefore using Lemma~\ref{lem:billingsley} we deduce
\[
\dpsi{K(D)\cap A} = \sum_{d\in D} p(d) \log(p(d)) + (1-\theta/H) \sum_{d\in D} p(d) \log(r(d)),
\]
and hence we obtain a lower bound for the parabolic dimension of $K(D)$
\[
\dpsi{K(D)} \geq \sum_{d\in D} p(d) \log(p(d)) + (1-\theta/H) \sum_{d\in D} p(d) \log(r(d)).
\]
Maximizing the right hand side of the above inequality over all probability measures $(p(d))$ gives that the maximizing measure is
\begin{align}\label{eq:probmeas}
p(d) = \frac{1}{Z}r(d)^{\theta/H-1} \quad \text{and} \quad Z = \sum_{d\in D} r(d)^{\theta/H-1} = \sum_{j=0}^{m-1}r(j)^{\theta/H}.
\end{align}
This choice of probability measure immediately gives
\[
\dpsi{K(D)} \geq H\log_m\sum_{j=0}^{m-1}r(j)^{\theta/H},
\]
and hence it remains to prove the upper bound.
From now we fix the choice of probability measure as in~\eqref{eq:probmeas}. We define
\[
S_k(x,y) = \sum_{\ell=1}^{k} r(y_\ell).
\]
Using~\eqref{eq:probmeas} we can rewrite~\eqref{eq:logmu} as follows
\begin{align*}
\log(\mu(Q_k(x,y)))& = \sum_{\ell=1}^{k} \log\left(\frac{1}{Z} r(y_\ell)^{\theta/H-1} \right) + \sum_{\ell=1}^{k} \log(r(y_\ell)) - \sum_{\ell=1}^{\lfloor \theta k/H\rfloor} \log(r(y_\ell) ) \\
&= -k \log(Z) + (\theta/H-1) S_k(x,y) + S_k(x,y) - S_{\lfloor \theta k/H\rfloor
}(x,y).
\end{align*}
Therefore
\begin{align}
\label{eq:integerpart}\frac{H}{\theta k} \log(\mu(Q_k(x,y)))+ \frac{H}{\theta} \log(Z) = \frac{S_k(x,y)}{k} - \frac{S_{\lfloor \theta k/H\rfloor}(x,y)}{\theta k/H}.
\end{align}
We can write the right hand side as follows
\begin{align*}
\frac{S_k(x,y)}{k} - \frac{S_{\lfloor \theta k/H\rfloor}(x,y)}{\theta k/H} = \frac{S_{\lfloor k\rfloor}(x,y)}{\lfloor k\rfloor} - \frac{S_{\lfloor \theta k/H\rfloor}(x,y)}{\lfloor \theta k/H\rfloor}\left(1- \frac{\{\theta k/H\}}{\theta k/H} \right),
\end{align*}
where for all $x$ we write $\{x\} = x-\lfloor x\rfloor$. Now we can sum the right hand side above over all $k=H/\theta,(H/\theta)^{2},\ldots$ and hence we get a telescoping series and a convergent one, since $(S_\ell/\ell)$ is bounded and $\theta/H <1$. In this way we get
\begin{align*}
\limsup_{k\to \infty} \left( \frac{S_k(x,y)}{k} - \frac{S_{\lfloor \theta k/H\rfloor}(x,y)}{\theta k/H} \right) \geq 0,
\end{align*}
since otherwise the sum of these differences would converge to~$-\infty$.
Hence, from~\eqref{eq:integerpart} we deduce
\begin{align*}
\liminf_{k\to \infty} \frac{\log(\mu(Q_k(x,y)))}{\log(m^{-k/H})} \leq H\log_m(Z)
\end{align*}
and applying now Lemma~\ref{lem:billingsley} we immediately conclude
\[
\dpsi{K(D)} \leq H\log_m(Z)
\]
and this finishes the proof of the theorem.
\end{proof}

\begin{proof}[\textbf{Proof of Corollary~\ref{cor:selfaffine}}]

The statement of the corollary follows immediately from Remark~\ref{rem:secondterm} and
Lemma~\ref{lem:self-affine}.
\end{proof}

We now proceed to prove Corollary~\ref{cor:strictineq}. To this end we first define a self-affine set $K$ and then show that there exists a function $f:[0,1]\to [0,1]$ which is H\"older continuous with parameter~$\log 2/\log 6$ and satisfies $\gr{f} = K$.

We start by defining the self-affine set that corresponds to the patterns $A$ and $B$ given by the matrices
\[
A =  \left( \begin{matrix} 0&0&0&0&0&1\\ 1&1&1&1&1&0 \end{matrix}\right) \quad \text{and} \quad B = \left( \begin{matrix} 1&0&0&0&0&0\\ 0&1&1&1&1&1 \end{matrix}\right).
\]
Let $\Q_0 = \{[0,1]^2\}$ be the set containing the rectangles of the $0$-th generation. To each rectangle in $\Q_0$ we assign label $A$. Suppose we have defined the collection $\Q_j$ and assigned labels to the rectangles in $\Q_j$. Then we subdivide each rectangle $R_j$ in $\Q_j$ into $12$ equal closed rectangles of width $6^{-(j+1)}$ and height $2^{-(j+1)}$. If the label assigned to $R_j$ is $A$ (resp.\ $B$), then in the subdivision we keep only those rectangles that correspond to the pattern $A$ (resp.\ $B$). If the label of $R_j$ is $A$, then to the rectangles that we kept we assign labels $A, B, A, B, A, A$ going
from left to right. If the label of $R_j$ is $B$, then to the rectangles that we kept we assign labels $B, B, A, B, A, B$ again going from left to right. The collection $\Q_{j+1}$ consists of those rectangles that we kept in the above procedure. Continuing indefinitely gives a compact set which we will denote $K$. The patterns~$A$ and~$B$ and the labels used in each iteration are depicted in Figure~\ref{fig:patterns} and the first four approximations to the set $K$ are shown in Figure~\ref{fig:K} in the Introduction.

\begin{claim}\label{cl:f}
There exists a function $f:[0,1]\to[0,1]$ such that $\gr{f} = K$. Moreover, $f$ is H\"older continuous with parameter $\theta = \log 2/\log 6$ and is not H\"older continuous with parameter~$\theta'$ for any~$\theta'>\theta$.
\end{claim}

 \begin{proof}[\textbf{Proof}]

 For every $x\in [0,1]$ let $x = \sum_{i=1}^{\infty} x_i 6^{-i}$ with $x_i\in \{0,1,2,3,4,5\}$ be its expansion in base~$6$. Note that if $x = k6^{-i}$ for some $k\in \{0,1,\ldots, 6^i\}$, then $x$ has two different expansions in base~$6$; one with an infinite number of $0$'s and one with an infinite number of $5$'s. To define the function $f$ we consider the expansion with the infinite number of $0$'s. We now define a sequence $(y_i)$ corresponding to the sequence $(x_i)$, where $y_i\in \{0,1\}$. For each rectangle $R\in \Q_j$ we consider the interval of the~$j$-th generation which is the projection of $R$ on $[0,1]$. This way we obtain a partition of $[0,1]$ into disjoint subintervals of length~$6^{-j}$ in generation $j$.

To determine $y_j$ we find the interval of the $j$-th generation where $x$ belongs to. If the pattern used in the rectangle of the $j$-th generation that corresponds to this interval is $A$, then if $x_j\neq 5$, we set $y_j = 0$, otherwise we set $y_j=1$.  If the pattern used is $B$, then if $x_j\neq 1$, we set $y_j = 0$, otherwise we set $y_j=1$. We finally define
\[
f(x) = \sum_{i=1}^{\infty} y_i 2^{-i}.
\]
It is now clear that $\gr{f} = K$. It remains to show the H\"older property.

We first argue that the definition of $f$ remains unchanged if we do not require for the representation of $x$ to have an infinite number of $0$'s. Suppose that $x$ lies on a dividing line of the $i$-th generation. Then the first $i$ digits of $x$ are independent of the representation. Thus the first $i$ digits of $f(x)$ are also independent. Then there are several cases. We illustrate four of them in Figure~\ref{fig:zoom}. In Figure~\ref{fig:A} the labels of the two rectangles above $x$ from left to right are $A, B$. This means that $y_{i+1}=1$ and this is independent of the representation. In the case of Figure~\ref{fig:AA} the two rectangles from left to right are assigned $A,A$. In the representation from the left $y_{i+1} = 0$ and from the right $y_{i+1}' =0$. In the next generations $y_{i+k}= 1$ and $y_{i+k}'=0$ for all $k\geq 2$. This now implies that $f(x)$ is independent of the representation in this case. The other cases follow similarly.

\begin{figure}
\label{fig:zoom}
\begin{center}
\subfigure[Pattern $A$]{\label{fig:A}
\includegraphics[scale=0.5]{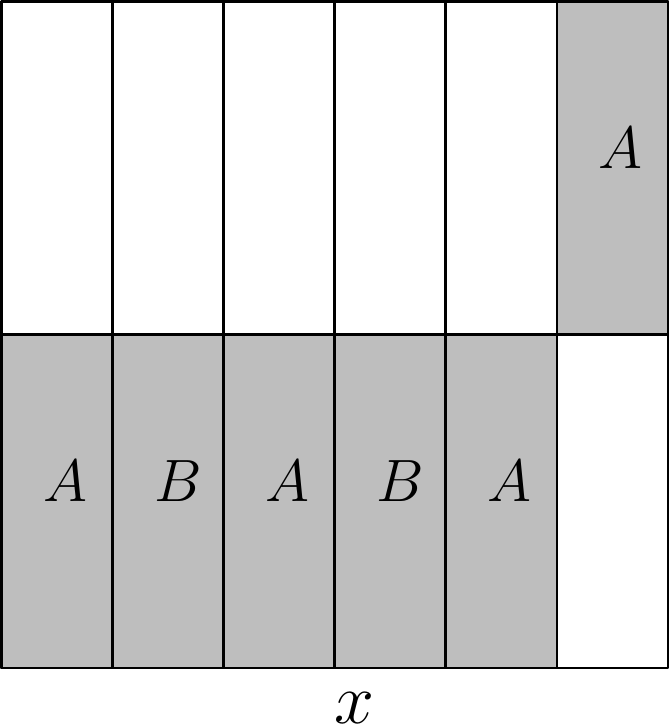}
}
\subfigure[Pattern $A$]{\label{fig:AA}
\includegraphics[scale=0.5]{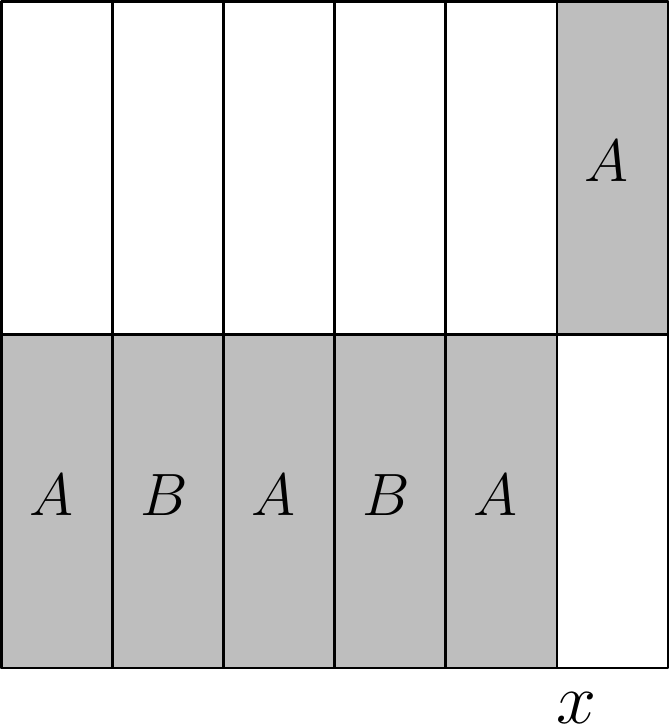}
}
\subfigure[Pattern $B$]{\label{fig:B}
\includegraphics[scale=0.5]{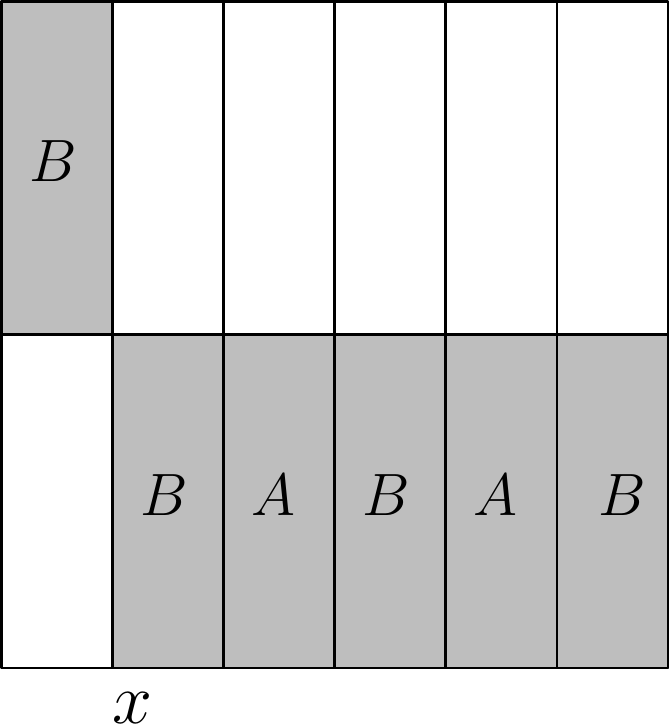}
}
\subfigure[Pattern $B$]{\label{fig:BB}
\includegraphics[scale=0.5]{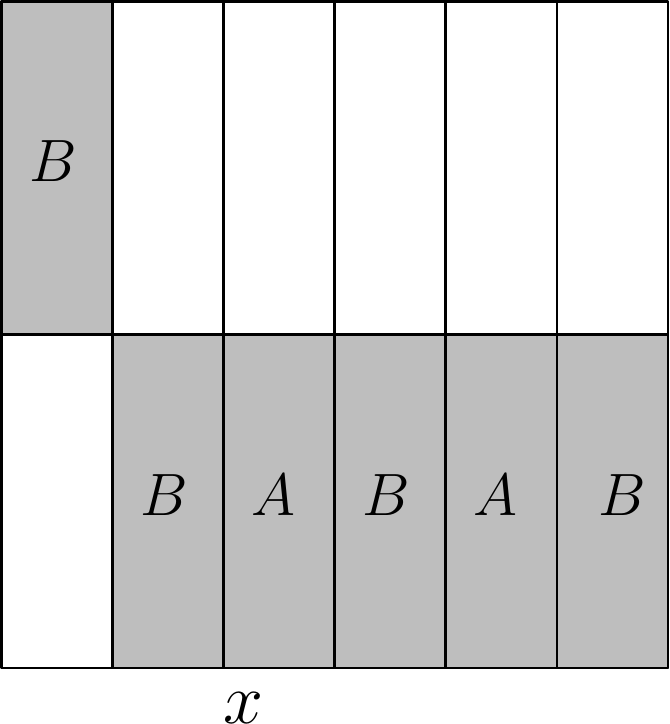}
}
\caption{Patterns $A, B$ and location of $x$ on the dividing line}
\end{center}
\end{figure}

It now remains to show that $f$ is H\"older continuous. Let $x$ and $x'$ satisfy
\[
6^{-k} < |x-x'| \leq 6^{-k+1} \quad \text{and} \quad x_i = x_i', \quad  \forall i\leq k.
\]
Then by the construction of $f$ it follows that $y_i = y_i'$ for all $i\leq k$, and hence
\[
|f(x) - f(x')| =\left|\sum_{i=1}^{\infty} (y_i - y_i') 2^{-i} \right| \leq \sum_{i=k+1}^{\infty} 2^{-i} = 2^{-k} = 6^{-\theta k} \leq |x-x'|^{\theta}.
\]
If $x,x'$ satisfy $6^{-k}<|x-x'|\leq 6^{-k+1}$ but disagree in the first $k$ digits, then let $x_0 = \ell 6^{-k}$ be the unique point of the $k$-th subdivision that agrees with $x$ and $x'$ in the first $k$ digits if we consider its two representations in base $6$. Then by the above argument it follows that
\[
|f(x)-f(x_0)| \leq |x-x_0|^\theta \quad \text{and} \quad |f(x') - f(x_0)| \leq |x'-x_0|^\theta.
\]
Therefore by the triangle inequality we immediately get that
\[
|f(x) - f(x')| \leq 2|x-x'|^\theta
\]
and this proves that $f$ is H\"older continuous with parameter $\theta$.

We note that $f$ is not H\"older continuous for any $\theta'>\theta$. Indeed, let $(x_n^{(k)})_k, (u_n^{(k)})_k$ be two sequences indexed by~$k$ such that $x_n^{(k)} = u_n^{(k)}$ for all $n\neq k$ and $x_k^{(k)} = 0$ and $u_k^{(k)}=5$ and let the rectangle of generation $k-1$ where $x=x^{(k)}$ and $u=u^{(k)}$ belong to have label $A$. Then it is easy to see that $y_k = 0$ and $v_k = 1$, where $f(x) =\sum_i y_i2^{-i}$ and $f(u)=\sum_i v_i2^{-i}$. Therefore
\[
|f(x) - f(u)| = |y_{k} - v_k| 2^{-k} =6^{-\theta k},
\]
and hence $f$ cannot be H\"older continuous for any $\theta'>\theta$.
 \end{proof}

\begin{proof}[\textbf{Proof of Corollary~\ref{cor:strictineq}}]

We first explain how we can adapt the proof of Lemma~\ref{lem:self-affine} in order to get the parabolic dimension of $K$, since the patterns used are not the same in each iteration as was the case there. We only outline where the two proofs differ.

Let $D_1 = \{(0,0), (0,1), (0,2), (0,3), (0,4), (1,5)\}$  and  $D_2 = \{(1,0), (0,1), (0,2), (0,3), (0,4), (0,5)\}$ correspond to patterns $A$ and~$B$ respectively.
We define two probability distributions on $D_1$ and on $D_2$. Let $p>0$ and $q>0$ satisfy $5p+q=1$. Then we let $p_1(x,0) =  p$ for all $x\neq 5$ and $p_1(5,1) = q$.  This is a distribution on $D_1$. We also let $p_2(0,1) = q$ and $p_2(x,0) = p$ for $x\neq 0$. This is a distribution on~$D_2$. We notice that both distributions only depend on the second coordinate and give the same values to this coordinate. We now generate $(\xi_i^1,\xi_i^2)_{i\geq 1}$ an i.i.d.\ sequence from $p_1$ and independently $(\zeta_i^1,\zeta_i^2)_{i\geq 1}$ an i.i.d.\ sequence from $p_2$. We sample $(x,y) \in K$ by sampling the digits. Namely, $(x_1,y_1)= (\xi_1^1,\xi_1^2)$ and then iteratively depending on the history of the process we set either $(x_i,y_i) = (\xi_{r(i)}^1,\xi_{r(i)}^2)$
or $(x_i,y_i) = (\zeta_{i-r(i)}^1,\zeta_{i-r(i)}^2)$,  where $r(i)$ is the number of times that we have used the distribution $p_1$.  Then if $\mu$ is the measure induced by these distributions we get for $(x,y)\in K$ and $Q_{k}(x,y)$ as defined in Lemma~\ref{lem:self-affine}
\begin{align*}
\mu(Q_k(x,y)) = \prod_{i=1}^{k} w(x_i,y_i) \prod_{j=\lfloor 2\theta k\rfloor +1}^{k} r(y_j),
\end{align*}
where $w(x_i,y_i)$ is either equal to $p_1(x_i,y_i)$ or to $p_2(x_i,y_i)$ and $\theta = \log 2/\log 6$. By the construction above it easily follows that $w(x_i,y_i)$ is an i.i.d.\ sequence that takes the value $p$ with probability $5p$ and the value $q$ with probability $q$. By the strong law of large numbers we then deduce that for $\mu$-almost every $(x,y)$
\begin{align*}
\lim_{k\to \infty} \frac{1}{k} \log \left(\mu(Q_k(x,y)) \right) = 5p\log p + q\log q + (1-2\theta) 5p\log 5.
\end{align*}
Now the rest of the proof follows in exactly the same way as the proof of Lemma~\ref{lem:self-affine} to finally give
\begin{align}\label{eq:dimf}
\dpsi{\gr{f}} = \frac{1}{2} \log_2(5^{2\theta} + 1),
\end{align}
where we used $H=1/2$ for the Brownian motion.
Let $f:[0,1]\to [0,1]$ be the function of Claim~\ref{cl:f} which is H\"older continuous with exponent~$\theta$ and satisfies $\gr{f} = K$. Then from~\eqref{eq:dimf} and Corollary~\ref{cor:selfaffine} we immediately get
\[
\dim(\gr{B+f}) = \frac{\log_2\left(5^{2\theta} +1 \right) +1}{2}.
\]
Since we have
\[
\dim (\gr{f}) = \log_2\left(5^\theta+1\right),
\]
it follows that
\[
\dim(\gr{B+f}) > \max\{\dim (\gr{f}), 3/2\}
\]
and this concludes the proof.
\end{proof}

\section{Comparing dimensions of $\gr{B+f}$ when $\gr{f}$ is a self affine set}\label{sec:mink}

\begin{theorem}
Let $B$ be a standard Brownian motion in $\R$ and $n>m^2$. Let $D\subseteq \{0,\ldots,n-1\}\times \{0,\ldots, m-1\}$ be a pattern such that every row always contains a chosen rectangle (i.e. $r_j \ge 1$ for all $j \le m-1$) and every column contains exactly one chosen rectangle. Then there exists a function~$f$ with $\gr{f} = K(D)$ and
we have almost surely
\begin{align}\label{eq:minky}
\dim_M (\gr{B+f}) = \dim_M (\gr{f}) = 1 + \log_n \frac{n}{m}.
\end{align}
Moreover, if the $r_j$ are not all equal, then almost surely
\begin{align}\label{eq:strict}
\max\{\dim(\gr{B}), \dim(\gr{f})\}  < \dim(\gr{B+f}) < \dim_M(\gr{B+f}) .
\end{align}
\end{theorem}

\begin{proof}[\textbf{Proof}]

Note that the function $f$ can be made c\`adl\`ag without affecting $\dim_M\gr{B+f}$ and $\dim_M\gr{f}$. Then we can apply~\cite[Theorem~1.7]{CPS12} to get that almost surely
\begin{align}\label{eq:lowerbdgrf}
\dim_M(\gr{B+f}) \geq \dim_M(\gr{f}).
\end{align}
It only remains to prove the upper bound. We follow McMullen's proof~\cite{McMullen} for the calculation of the Minkowski dimension of~$\gr{f}$. First notice that $\theta=\log_n m <1/2$.

Consider a rectangle of the $j$-th generation of the construction of $\gr{f}$ with size $n^{-j}\times m^{-j}$. Then it is of the form $R=[pn^{-j},(p+1)n^{-j}]\times [qm^{-j},(q+1)m^{-j}]$. By the H\"older property of Brownian motion it follows that for $\zeta>0$ there exists a constant $C$ such that almost surely for all $s,t\in [0,1]$ we have
\begin{align}\label{eq:holderpr}
|B_t-B_s| \leq C|t-s|^{1/2-\zeta}.
\end{align}
When $\gr{f}$ is perturbed by Brownian motion, then the above rectangle becomes
\[
R'=[pn^{-j},(p+1)n^{-j}]\times [qm^{-j}+B_{pn^{-j}}-Cn^{-j(1/2-\zeta)},(q+1)m^{-j} + B_{pn^{-j}} +Cn^{-j(1/2-\zeta)}].
\]
This means that if $(t,f(t)) \in R$, then by~\eqref{eq:holderpr} we have $(t,B_t+f(t)) \in R'$. If $\theta = \log_n m$, then the rectangle $R'$ requires $m^{j-[\theta j]}$ squares of side $n^{-j}$ to cover it, since $\theta<1/2$. Therefore the number of squares of side $n^{-j}$ needed to cover $\gr{B+f}$ is at most $|D|^j m^{j-[\theta j]}$. Taking logarithms and then the limit as $j\to \infty$ we obtain that almost surely
\begin{align*}
\dim_M(\gr{B+f}) \leq \lim_{j\to\infty} \frac{\log |D|^j m^{j-[\theta j]}}{\log n^j} = 1 + \log_n\frac{|D|}{m} = \dim_M(\gr{f})
\end{align*}
and this together with~\eqref{eq:lowerbdgrf} concludes the proof of (\ref{eq:minky}).

It remains to prove~\eqref{eq:strict}. By Cauchy-Schwartz we have
$$
\sum_{j=0}^{m-1} r_j^\theta \le \Bigl(m \sum_{j=0}^{m-1} r_j^{2\theta} \Bigr)^{1/2} \,.
$$
Therefore almost surely we get
\begin{align}\label{eq:strict1}
\dim(\gr{f}) =\log_m(\sum_{j=0}^{m-1} r_j^\theta) \le \frac{1}{2}\Bigl(1+ \log_m(\sum_{j=0}^{m-1} r_j^{2\theta}) \Bigr)= \dim(\gr{B+f}).
\end{align}
Since $2\theta<1$, we have $\sum_{j=0}^{m-1} r_j^{2\theta} > \sum_{j=0}^{m-1} r_j n^{2\theta-1} =n^{2\theta}$. Thus
$$
\dim(\gr{B+f}) > \frac{ 1+\log_m (n^{2\theta})}{2}=\frac{3}{2}=\dim(\gr{B}) \quad a.s.
$$
and together with (\ref{eq:strict1}), this proves the first inequality in (\ref{eq:strict}).

If the $r_j$ are not all equal, then by Jensen's inequality we get
 $$
 \frac{1}{m} \sum_{j=0}^{m-1} r_j^{2\theta} <\Bigl(\frac{1}{m}\sum_{j=0}^{m-1} r_j \Bigr)^{2\theta}  =(n/m)^{2\theta} \,,
$$
whence
$$
\dim(\gr{B+f}) < \frac{ 1+\log_m \Bigl(m (n/m)^{2\theta}\Bigr)}{2}=2-\theta \,,
$$
establishing the second inequality in (\ref{eq:strict}).
\end{proof}

\bibliographystyle{plain}
\bibliography{biblio}

\begin{thebibliography}{10}

\bibitem{BayHeu}
F.~Bayart and Y.~Heurteaux.
\newblock On the {H}ausdorff dimension of graphs of prevalent continuous
  functions on compact sets.
\newblock {In Further Developments in Fractals and Related Fields (eds. Barral,
  J; Seuret, S), Birkhauser 2013.}

\bibitem{Bedford_thesis}
T.~Bedford.
\newblock {\em {C}rinkly curves, {M}arkov partitions and box dimensions in
  self-similar sets}.
\newblock PhD thesis, 1984.
\newblock Ph.D. Thesis, University of Warwick.

\bibitem{Billingsley}
Patrick Billingsley.
\newblock Hausdorff dimension in probability theory.
\newblock {\em Illinois J. Math.}, 4:187--209, 1960.

\bibitem{Cajar}
Helmut Cajar.
\newblock {\em Billingsley dimension in probability spaces}, volume 892 of {\em
  Lecture Notes in Mathematics}.
\newblock Springer-Verlag, Berlin, 1981.

\bibitem{Carlesson}
Lennart Carleson.
\newblock {\em Selected problems on exceptional sets}.
\newblock Van Nostrand Mathematical Studies, No. 13. D. Van Nostrand Co., Inc.,
  Princeton, N.J.-Toronto, Ont.-London, 1967.

\bibitem{CPS12}
P.~H.~A. Charmoy, Y.~Peres, and P.~Sousi.
\newblock Minkowski dimension of brownian motion with drift, 2012.
\newblock {arXiv:1208.0586}.

\bibitem{falconer88}
K.~J. Falconer.
\newblock The {H}ausdorff dimension of self-affine fractals.
\newblock {\em Math. Proc. Cambridge Philos. Soc.}, 103(2):339--350, 1988.

\bibitem{Kahane}
Jean-Pierre Kahane.
\newblock {\em Some random series of functions}, volume~5 of {\em Cambridge
  Studies in Advanced Mathematics}.
\newblock Cambridge University Press, Cambridge, second edition, 1985.

\bibitem{Kechris}
Alexander~S. Kechris.
\newblock {\em Classical descriptive set theory}, volume 156 of {\em Graduate
  Texts in Mathematics}.
\newblock Springer-Verlag, New York, 1995.

\bibitem{KhosXiao}
D.~{Khoshnevisan} and Y.~{Xiao}.
\newblock {Brownian motion and thermal capacity}.
\newblock {\em ArXiv e-prints}, April 2011.

\bibitem{McMullen}
Curt McMullen.
\newblock The {H}ausdorff dimension of general {S}ierpi\'nski carpets.
\newblock {\em Nagoya Math. J.}, 96:1--9, 1984.

\bibitem{Perescarpet}
Yuval Peres.
\newblock The self-affine carpets of {M}c{M}ullen and {B}edford have infinite
  {H}ausdorff measure.
\newblock {\em Math. Proc. Cambridge Philos. Soc.}, 116(3):513--526, 1994.

\bibitem{PS10}
Yuval Peres and Perla Sousi.
\newblock Brownian motion with variable drift: 0-1 laws, hitting probabilities
  and {H}ausdorff dimension.
\newblock {\em Math. Proc. Cambridge Philos. Soc.}, 153(2):215--234, 2012.

\bibitem{solomyak98}
Boris Solomyak.
\newblock Measure and dimension for some fractal families.
\newblock {\em Math. Proc. Cambridge Philos. Soc.}, 124(3):531--546, 1998.

\bibitem{TaylorWatson}
S.~J. Taylor and N.~A. Watson.
\newblock A {H}ausdorff measure classification of polar sets for the heat
  equation.
\newblock {\em Math. Proc. Cambridge Philos. Soc.}, 97(2):325--344, 1985.

\end{thebibliography}

\end{document}